\documentclass[10pt]{amsart}

\usepackage{amsfonts}
\usepackage{amscd}
\usepackage{amsmath}
\usepackage{amsthm}
\usepackage{amssymb}
\usepackage{amsxtra}

\input xy
\xyoption{all}

\newcommand{\bL}{\mathbf{L}}
\newcommand{\bR}{\mathbf{R}}

\newcommand{\caC}{{\mathcal C}}
\newcommand{\caD}{{\mathcal D}}

\newcommand{\caE}{{\mathcal E}}

\newcommand{\caN}{{\mathcal N}}

\newcommand{\caO}{{\mathcal O}}

\newcommand{\caT}{{\mathcal T}}

\newcommand{\caF}{{\mathcal F}}

\newcommand{\caQ}{{\mathcal Q}}
\newcommand{\caK}{{\mathcal K}}

\newcommand{\GG}{\mathsf{G}}

\newcommand{\integers}{\mathbf{Z}}
\newcommand{\bG}{\mathbb{G}}
\newcommand{\bH}{\mathbb{H}}
\newcommand{\bQ}{\mathbb{Q}}

\newcommand{\Hom}{\mathrm{Hom}}
\newcommand{\Ho}{\mathsf{Ho}}

\newcommand{\sSet}{\mathrm{\bf sSet}}

\newcommand{\colim}{\mathrm{colim}}
\newcommand{\hocolim}{\mathrm{hocolim}}
\newcommand{\hofib}{\mathrm{hofib}}
\newcommand{\cofib}{\mathrm{cofib}}
\newcommand{\Totcof}{\mathrm{Totcof}}

\newcommand{\eff}{{\mathit{eff}}}

\newcommand{\op}{\mathrm{op}}

\newcommand{\Mod}{\mathrm{Mod}}

\newcommand{\Spec}{\mathsf{Spec}}
\newcommand{\X}{\mathcal{X}}
\newcommand{\Top}{\mathrm{Top}}
\newcommand{\SlMGL}{\mathrm{(SlMGL)}}
\newcommand{\LQ}{\mathsf{L}\mathbf{Q}}
\newcommand{\PLQ}{\mathsf{PL}\mathbf{Q}}

\theoremstyle{theoremstyle}
\newtheorem{theorem}{Theorem}[section]
\newtheorem*{theorem*}{Theorem}
\newtheorem{lemma}[theorem]{Lemma}
\newtheorem{proposition}[theorem]{Proposition}
\newtheorem*{proposition*}{Proposition}
\newtheorem{corollary}[theorem]{Corollary}
\newtheorem*{corollary*}{Corollary}

\newtheorem*{conjecture*}{Conjecture}

\newtheorem{definition*}{Definition}
\newtheorem{remark}[theorem]{Remark}
\newtheorem{remark*}{Remark}

\newcommand{\A}{\mathbf{A}}
\newcommand{\D}{\mathsf{D}}
\newcommand{\Q}{\mathsf{Q}}
\newcommand{\I}{\mathsf{i}}

\newcommand{\SH}{\mathbf{SH}}
\newcommand{\LL}{\mathsf{L}}
\newcommand{\PP}{\mathbf{P}}

\newcommand{\BGL}{\mathsf{B}\mathbf{GL}}
\newcommand{\KGL}{\mathsf{KGL}}
\newcommand{\MGL}{\mathsf{MGL}}

\newcommand{\MU}{\mathsf{MU}}
\newcommand{\BP}{\mathsf{BP}}
\newcommand{\MZ}{\mathsf{M}\mathbf{Z}}

\newcommand{\Sm}{\mathsf{Sm}}
\newcommand{\EE}{\mathsf{E}}

\newcommand{\unit}{\mathbf{1}}

\newcommand{\Z}{\mathbf{Z}}

\newcommand{\Thom}{\mathsf{Th}}

\newcommand{\einscat}{\underline{1}}
\newcommand{\poauf}{[\![}
\newcommand{\pozu}{] \! ]}

\renewcommand{\S}{\mathbb{S}}

\renewcommand{\smash}{\wedge}

\title{Relations between slices and quotients of the
algebraic cobordism spectrum}
\author{Markus Spitzweck}
\date{\today}
\begin{document}

\pagestyle{plain}
\maketitle

\begin{abstract}
We prove a relative statement about the slices of the algebraic cobordism spectrum.
If the map from $\MGL$ to a certain quotient of $\MGL$ introduced by Hopkins and Morel
is the map to the zero-slice then a relative version of Voevodsky's conjecture
on the slices of $\MGL$ holds true. We outline the picture for $K$-theory and
rational slices.
\end{abstract}

\tableofcontents

\section{Introduction}
\noindent
In this paper we discuss certain aspects of the slice filtration
of the algebraic cobordism spectrum $\MGL$.
The slice filtration was introduced in \cite{voe-slice}. It is a filtration on any motivic spectrum
which can be thought of as an analogue of the Postnikov
tower of a topological spectrum.
We discuss the relation between the slice conjecture on $\MGL$
\cite[Conjecture (5)]{voe-slice}
and quotients of $\MGL$. The conjecture describes the slices in terms of the motivic Eilenberg MacLane
spectrum and the topological coefficients $\MU_*$.
The quotients are defined using the classifying map
$\LL_* \to \MGL_{**}$ of the formal group law induced by the canonical  orientation of $\MGL$
and canonical generators $x_i$ of $\LL_*$. Here $\LL_*$ denotes the graded Lazard
ring. In topology it is well known that
the quotient $\MU/(x_1,x_2, \ldots) \MU$ is isomorphic
to the Eilenberg MacLane spectrum on the integers. This follows essentially
from Quillen's theorem $\LL_* \cong \MU_*$
and the particular structure of $\LL_*$.
In motivic homotopy theory no direct analog of this argument seems
to work. In particular, the filtration on $\MGL$ obtained by dividing out the $x_i$
(more precisely ideals of $\LL_*$ consisting of elements of degree greater than a given bound)
is not directly reflected on the homotopy groups of $\MGL$,
instead this filtration is conjecturally the slice filtration
of $\MGL$.
We show that if this holds on the zeroth level then it holds in all levels.
The statement is of a purely homotopy theoretic nature
and could be formulated in any context of highly structured ring spectra.

Over perfect fields the zero slice of the sphere spectrum is known \cite{voevodsky-zero-slice}, \cite{levine-htp}.
An effectivity result for $\MGL$ implies that $\MGL$ has the same zero slice, see Corollary (\ref{s0-isom}).

Our main statement implies that if the quotient of
$\MGL$ by all the $x_i$ coincides with this zero slice then
\cite[Conjecture (5)]{voe-slice} holds, see Corollary (\ref{per-field}).
As a consequence, under the same assumption, slices of all
Landweber exact spectra are given in terms of the Landweber coefficients and
the motivic Eilenberg Maclane spectrum, see \cite{spitzweck-slice}.

The author was inspired
by the course notes \cite{mot-htp}.

Here is an overview of the sections.
In the first section we show effectivity and cellularity results for the
algebraic cobordism spectrum.
Theorem (\ref{zero-slice}) states that the cofiber of the unit map of $\MGL$ lies in a certain subcategory
of the stable motivic homotopy category spanned by positive Tate-spheres.
As a corollary we obtain the effectivity of $\MGL$ and
a proof of the observation in \cite{voe-slice} that the zero slices
of the sphere spectrum and $\MGL$ coincide.
These results can be viewed as refinements
to the cellularity results of \cite{dugger-isaksen}.
The effectivity of $\MGL$ was implicitely assumed in \cite{spitzweck-slice}.

The second paragraph deals with our main observation that if the
Hopkins-Morel quotient of $\MGL$ is the zero slice of $\MGL$ then a relative version
of Voevodsky's conjecture \cite{voe-slice} on slices holds true, Theorem (\ref{rel-voe-conj}). This is closely related
to the work of Hopkins and Morel on the spectral sequence for $\MGL$ in terms
of motivic cohomology, see \cite{levine-comparison}. In particular it is an unpublished result
due to Hopkins and Morel that over fields of characteristic $0$ the quotient of $\MGL$ by the
$x_i$ is isomorphic to the motivic Eilenberg MacLane spectrum, \cite{mot-htp}. It thus
follows form Corollary (\ref{per-field}) that \cite[Conjecture (5)]{voe-slice} holds
over fields of characteristic $0$ assuming the Hopkins-Morel isomorphism.
As above we get under the same assumption the structure of the slices of Landweber spectra.

We note that the idea of using filtrations of the ideal of the Lazard ring
spanned by elements of positive degree in a homotopy way goes back to \cite{mot-htp}.

The third paragraph covers the relationship of quotients of $\MGL$ to the
algebraic $K$-theory spectrum. The main statement is also contained in \cite{mot-htp}.

In the last paragraph we consider rational slices.
Due to the rational splitting of $\MGL$
this simplifies to understanding the rational Landweber theory $\LQ$ for the additive
formal group law over the rationals, see \cite{NSO1}. We obtain these assertions over regular base schemes
by comparing the Landweber decomposition of $\KGL_\bQ$ and
a decomposition obtained in \cite{riou-these}.

{\bf Acknowledgements.} I thank the lecturers of \cite{mot-htp}. I thank Niko Naumann, Paul Arne {\O}stv{\ae}r,
Ivan Panin and Oliver R\"ondigs
for inspiring communications.
I thank Ulrich Bunke and Ansgar Schneider for helpful comments.

\section{Conventions}

Throughout the article we work over a Noetherian base scheme of finite Krull dimension
$S$. The stable motivic homotopy category over $S$ is denoted by $\SH(S)$. The standard spheres
are denoted by $S^{p,q} \cong S_s^{\wedge (p-q)} \wedge \bG_m^{\wedge q}$.
The category of smooth schemes over $S$ is denoted by $\Sm/S$.
The tensor unit of $\SH(S)$, i.e. the sphere spectrum, is denoted by $\unit$.

The full subcategory of $\SH(S)$ of effective spectra is denoted by $\SH(S)^\eff$. It is the full localizing
triangulated subcategory generated by all $\Sigma_T^\infty X_+$, $X \in \Sm/S$.

The $i$-th slice of a motivic spectrum $\EE$ is denoted $s_i(\EE)$, see \cite{voe-slice}.

Throughout the text we will use the language of model categories.
An injective model structure will be a model structure e.g. on a diagram category where weak
equivalences and cofibrations are understood objectwise. Dually the projective model structure
has objectwise fibrations. For homotopy colimits and limits we refer to \cite{hirschhorn}.
For symmetric monoidal model categories we refer to \cite{hovey-book}.

\section{Effectivity of $\MGL$}

Let $\SH(S)_\caT$ be the full localizing triangulated subcategory of $\SH(S)$
spanned by $\{S^{p,q} | p,q \in \integers\}$, see \cite{NSO1}.
We let $\SH(S)_{\caT_{\ge 0}}$ the  full localizing triangulated subcategory of $\SH(S)$ spanned
by $\{S^{p,q} | p \in \integers, q \ge 0\}$.

\begin{theorem} \label{zero-slice}
 The cofiber of the unit map $\unit \to \MGL$ is contained in $\Sigma_T \SH(S)_{\caT_{\ge 0}}$.
\end{theorem}

\begin{corollary}
 We have $\MGL \in \SH(S)_{\caT_{\ge 0}}$. In particular $\MGL$ is an effective spectrum.
\end{corollary}
\begin{proof}
 Follows from $\unit \in \SH(S)_{\caT_{\ge 0}}$.
\end{proof}

\begin{corollary} \label{s0-isom}
 The unit map $\unit \to \MGL$ induces an isomorphism
 $$s_0(\unit) \overset{\cong}{\longrightarrow} s_0(\MGL).$$
\end{corollary}
\begin{proof}
 The functor $s_0$ is triangulated and $s_0(X)=0$ for any $X \in \Sigma_T \SH(S)^\eff$.
\end{proof}

We start with preparations for the proof of theorem (\ref{zero-slice}).

\begin{lemma}\label{two-tri}
 Let $r$ be an integer and let
 $$S^{2r,r} \to X \to Y \to S^{2r,r}[1]$$
and
 $$X \to Z \to W \to X[1]$$
be two triangles in $\SH(S)$.
Suppose $Y,W \in \Sigma_T^{r+1} \SH(S)_{\caT_{\ge 0}}$.
Then the cofiber of $S^{2r,r} \to Z$ is in $\Sigma_T^{r+1} \SH(S)_{\caT_{\ge 0}}$.
\end{lemma}
\begin{proof}
 The cofiber of $S^{2r,r}\to Z$ is an extension of $W$ by $Y$.
\end{proof}

\begin{lemma} \label{thom-s}
 Let $i \colon Z \to X$ be a closed immersion in $\Sm/S$ and $\caE$
 a vector bundle over $X$. Let $U= X \setminus Z$, and denote the restrictions
 of $\caE$ to $Z$ and $U$ by $\caE_Z$ and $\caE_U$, respectively.
 Let $\caN$ be the normal bundle of $i$. Then the cofiber of
$\Thom(\caE_U) \to \Thom(\caE)$ is canonically isomorphic to
$\Thom(\caE_Z \oplus \caN)$ in the $\A^1$-homotopy category.
\end{lemma}
\begin{proof}
 This follows from the fact that this cofiber is (as Zariski sheaves)
isomorphic to $\caE/(\caE \smallsetminus j(Z))$, $j$ the composition of $i$ followed by the zero section of $\caE$,
and that the normal bundle of $j$ is $\caE_Z \oplus \caN$.
\end{proof}

We first quote some facts from \cite{NSO1} about finite Grassmannians. We let
$\GG(n,d)$ be the scheme parameterizing locally free quotients of rank $d$ of the trivial
bundle of rank $n$.
There is a universal short exact sequence
\begin{equation}\label{65alpha}
\xymatrix{ 0\ar[r]& \caK_{n,d}\ar[r]& \caO_{\GG(n,d)}^n\ar[r]& \caQ_{n,d}\ar[r]& 0}
\end{equation}
of vector bundles on $\GG(n,d)$, and
letting $\caK_{n,d}'$ the dual of $\caK_{n,d}$,
the tangent bundle is given by
\begin{equation}
\label{652}
\caT_{\GG(n,d)}\cong \caQ_{n,d}\otimes \caK_{n,d}'.
\end{equation}
The map
\[
\xymatrix{ i:\GG(n,d)\ar@{^{(}->}[r] &  \GG(n+1,d+1)} \]
classifying $\caK_{n,d}\subseteq\caO_{\GG(n,d)}^n\hookrightarrow\caO_{\GG(n,d)}^{n+1}$
is a closed immersion.
From (\ref{652}) it follows that the normal bundle $\caN(i)$ of $i$ identifies with $\caK_{n,d}'$.
Next consider the composition on $\GG(n+1,d+1)$
\[
\xymatrix{
\alpha:\caO_{\GG(n+1,d+1)}^n\ar@{^{(}->}[r] & \caO_{\GG(n+1,d+1)}^{n+1}\ar[r] & \caQ_{n+1,d+1} }
\]
for the inclusion into the first $n$ factors.
The complement of the support of coker($\alpha$) is an open subscheme $U\subseteq\GG(n+1,d+1)$ and there is a map
$\pi:U\to\GG(n,d+1)$ classifying $\alpha|_U$.
It is easy to see that $\pi$ identifies with $\caQ_{n,d+1} \to \GG(n,d+1)$.
An argument with geometric points reveals that $U=\GG(n+1,d+1)\smallsetminus i(\GG(n,d))$.
Moreover the natural map $\iota \colon \GG(n,d+1) \to \GG(n+1,d+1)$ classifying the subbundle $\caK_{n,d+1} \oplus \caO \subset \caO^{n+1}$
is the zero section $ \GG(n,d+1)\to\caQ_{n,d+1}$ followed by the inclusion $\caQ_{n,d+1} \cong U \to\GG(n+1,d+1)$.

We summarize the above with a diagram:
\begin{equation}\label{653}
\xymatrix{
\GG(n,d) \ar@{^{(}->}[r]^-i & \GG(n+1,d+1) & U \ar@{_{(}->}[l] \ar[r]_-\pi & \ar@<-1ex>@/_/[l] & \hspace{-1cm} \GG(n,d+1).}
\end{equation}

We note that compositions of the morphisms $\iota$ yields a map $\overline{\iota} \colon \mathrm{pt} \cong \GG(d,d) \to
\GG(n,d)$ which we consider as the natural pointing of $\GG(n,d)$. Note that the unit of $\MGL$ is
induced via these maps.

\begin{proposition}
 Let $\caE$ be a vector bundle of rank $r$ over $\GG(n,d)$ which is a finite sum of copies of $\caK_{n,d}$, $\caK_{n,d}'$ and $\caO$.
 Then $\overline{\iota}^* \caE$ is canonically trivialized. Furthermore the cofiber of the map of suspension
spectra of Thom spaces $S^{2r,r} \to \Sigma^\infty \Thom(\caE)$ lies in
$\Sigma_T^{r+1} \SH(S)_{\caT_{\ge 0}}$.
\end{proposition}
\begin{proof}
 We prove the statement by induction on $n$. It clearly holds for $n=0$.
Suppose $n \ge 0$ and assume the statement holds for $n$. The statement holds for $\GG(n+1,0)$ and $\GG(n+1,n+1)$.
Let $0 \le d < n$. We prove the statement for $\GG(n+1,d+1)$. Let $\caE$ be a vector bundle on
$\GG(n+1,d+1)$ of the considered type. It is canonically trivialized over the pointing.
We consider the diagram (\ref{653}). By lemma (\ref{thom-s}) we get an induced exact triangle
\begin{equation}\label{555}
 \xymatrix{\Sigma^\infty\Thom(\caE_U) \ar[r] & \Sigma^\infty\Thom(\caE) \ar[r] &
\Sigma^\infty\Thom(\caE_{\GG(n,d)} \oplus \caK_{n,d}') \ar[r] & \Sigma^\infty\Thom(\caE_U)[1]}.
\end{equation}
We note $\caK_{n+1,d+1}|_{\GG(n,d)} \cong \caK_{n,d}$, hence $\caE_{\GG(n,d)}$ and $\caE_{\GG(n,d)} \oplus \caK_{n,d}'$
are vector bundles on $\GG(n,d)$ of the considered type.

Since $\iota^* \caK_{n+1,d+1} \cong \caK(n,d+1) \oplus \caO$ there is an induced map
$\Thom(\caK_{n,d+1} \oplus \caO) \to \Thom(\caK_{n+1,d+1})$ which factors through
$\Thom(\caK_{n+1,d+1}|_U)$. Note $\Thom(\caK_{n,d+1} \oplus \caO) \to \Thom(\caK_{n+1,d+1}|_U)$
and more generally $\Thom(\iota^* \caE) \to \Thom(\caE_U)$
are motivic weak equivalences: we cover $U$ by the opens for which $\alpha$ restricted to a fixed
subset of the summands of $\caO^n$ of size $d+1$ surjects onto $\caQ_{n+1,d+1}$. Those opens
are pullbacks from $\GG(n,d+1)$. We claim on such opens the situation trivializes completely.
Without loss of generality we can assume the first $d+1$ summands of $\caQ^n$ surject onto $\caQ_{n+1,d+1}$.
Then $\caK_{n+1,d+1}$ restricted to this open trivializes by projecting to the last $n-d$ summands of $\caO^{n+1}$.
This trivialization is compatible with the one over the corresponding open $V \subset \GG(n,d+1)$. Thus
$((\iota^*\caE)|_V)^\circ \to (\caE|_{\pi^{-1}(V)})^\circ$ is a motivic weak equivalence,
and the same holds for the map $(\iota^*\caE)^\circ \to \caE_U^\circ$ by a Mayer-Vietoris argument.
This shows the claim that the map of Thom spaces is also a motivic weak equivalence.

We can thus rewrite the sequence (\ref{555}) as
\begin{equation}
  \xymatrix{\Sigma^\infty\Thom(\iota^*\caE) \ar[r] & \Sigma^\infty\Thom(\caE) \ar[r] &
\Sigma^\infty\Thom(\caE_{\GG(n,d)} \oplus \caK_{n,d}') \ar[r] & \Sigma^\infty\Thom(\iota^*\caE)[1]}
\end{equation}
and use the fact that $\iota^* \caE$ is of the type considered for $\GG(n,d+1)$.
By induction hypothesis the cofiber of $S^{2r,r} \cong \Sigma^\infty \Thom(\overline{\iota}^* \iota^* \caE)
\to \Sigma^\infty \Thom(\iota^* \caE)$ lies in $\Sigma_T^{r+1} \SH(S)_{\caT_{\ge 0}}$.
Moreover again by induction hypothesis $\Sigma^\infty \Thom(\caE_{\GG(n,d)} \oplus \caK_{n,d}')
\in\Sigma_T^{r+j} \SH(S)_{\caT_{\ge 0}}$ with $j=n-d>0$.
Now the statement follows from lemma (\ref{two-tri}).
\end{proof}
\begin{proof}[Proof of Theorem (\ref{zero-slice})]
 We let $\BGL_n = \colim_d \GG(n+d,d)$, $\xi_n=\colim_d \caK_{n+d,d}$ the universal vector bundle.
Then $$\MGL=\hocolim_n \Sigma^{-2n,-n}\Sigma^\infty\Thom(\xi_n) \cong \hocolim_{n,d} \Sigma^{-2n,-n}\Sigma^\infty
\Thom(\caK_{n+d,d}).$$ The unit $\unit \to \MGL$ is induced via the maps $$\Sigma^{-2n,-n} \Sigma^\infty\Thom(\overline{\iota}^*
\caK_{n+d,d})\to \Sigma^{-2n,-n}\Sigma^\infty \Thom(\caK_{n+d,d}).$$ By proposition (\ref{555}) the cofibers
of these maps are in $\Sigma_T \SH(S)_{\caT_{\ge 0}}$. Since cofiber sequences
are compatible with homotopy colimits the claim follows.
\end{proof}

\section{Quotients and slices of $\MGL$}

Let $\MGL$ denote a fibrant and cofibrant
model as commutative $\S$-algebra of the algebraic cobordism spectrum. We work in the simplicial
version of the $\S$-modules of \cite{ekmm},
see \cite{spitzweck-thesis}. In particular
$\MGL$ is fibrant and cofibrant as (symmetric) $T$-spectrum.
We let $\Mod(\MGL)$ be the symmetric monoidal
category with weak unit of $\MGL$-modules. The homotopy category of $\Mod(\MGL)$
is denoted by $\caD_\MGL$. It is a closed tensor triangulated
category.

We denote $\varphi \colon \MU_* \to \MGL_*$
the canonical map and fix an isomorphism
$\MU_* \cong \integers[x_1,x_2, \ldots]$,
$|x_i| = i$ (we divide the usual topological
grading by $2$).

As in topology we can form the
quotient $\Q:=\MGL/(x_1,x_2, \ldots) \MGL$
by taking iterated cofibers of
multiplications by the $x_i$ in $\caD_\MGL$.

This quotient is well-defined up to isomorphism.

We give a construction of a more explicit
model of this quotient:

We pick once and for all the following data:
\begin{enumerate}
\item $S^{2i,i}=T^{\smash i}$, $i > 0$, a cofibrant model of the $(2i,i)$-sphere
in $T$-spectra, and denote the corresponding cofibrant
sphere $S^{2i,i} \wedge \MGL$ by $\Sigma^{2i,i} \MGL$.
\item a map $\Sigma^{2i,i} \MGL \to \MGL$
representing the element $\varphi_*(x_i) \in \MGL_{2i,i}$, also denoted by $x_i$.
\end{enumerate}

We let $I$ be the following category: the objects are (commutative) monomials
in the $x_i$, there is a unique map from
$M$ to $N$ if $N$ divides $M$; the monomial
$1$ is allowed. We let $I^\circ$ be the full
subcategory consisting of all non-constant
monomials. The subcategory $I_{\le 1}$ consists
of monomials containing each $x_i$ with
multiplicity at most $1$, $I_{\le 1}^\circ$
is the same category with the constant monomial
removed.

We let $\einscat$ be the category whose diagrams
are morphisms, i.e. two objects $0$ and $1$
and one non identity map. We let $\einscat^n
\subset I_{\le 1}$ be the inclusion
of the monomials containing only the
$x_i$, $i \le n$. Via these inclusions
$I_{\le 1}$ is the union of all the
$\einscat^n$.

We are going to define the following $I$-diagram
of $\MGL$-modules: a monomial $x_1^{k_1} \cdots x_n^{k_n}$, $k_i \ge 0$, is
sent to $(\Sigma^{2,1} \MGL)^{\wedge_\MGL k_1} \wedge_\MGL \ldots \wedge_\MGL (\Sigma^{2n,n} \MGL)^{\wedge_\MGL k_n}$.
The morphisms will be given by multiplications with the $x_i$, i.e. iterations
of the morrphisms $x_i$. We have to be careful  since in general the two possible maps $\Sigma^{2i,i} \MGL \wedge_\MGL \Sigma^{2i,i} \MGL \to \Sigma^{2i,i} \MGL$
given by applying the map $x_i$ either on the left or on the right and then composing with a unit morphism (note $\MGL$ only serves as a weak unit) do not
coincide in general. Therefore we make the convention that for a map $M \to N$ in $I$, $M=x_1^{k_1} \cdots x_n^{k_n}$, $N=x_1^{l_1} \cdots x_n^{l_n}$, $l_i \le k_i$,
we insert for any $1 \le i \le n$ the map $(\Sigma^{2i,i} \MGL)^{\wedge_\MGL k_1} \to (\Sigma^{2i,i} \MGL)^{\wedge_\MGL l_1}$ which applies
$x_i$ on the $k_i-l_i$ right most tensor factors of the source followed by appropriate unit maps.

We end up with a diagram of $\MGL$-modules, denoted $\D$.

We need the following lemma in which we denote by $\hofib$ the homotopy fiber of a map
between pointed simplicial sets.

\begin{lemma} \label{pullb-fib}
 Let
$$\xymatrix{X \ar[r] \ar[d] & Y \ar[d] \\
           Z \ar[r] & W}$$
be a diagram of pointed simplicial sets. Let $P$ be the homotopy pullback of the right lower triangle of the square.
Then the homotopy fiber of the natural map $X \to P$ is naturally equivalent to
$\hofib(\hofib(X \to Z) \to \hofib(Y \to W))$.
\end{lemma}
\begin{proof}
 Replacing the diagram with an injectively fibrant diagram
the statement follows from the corresponding strict statement.
\end{proof}

Let $\D_{\le 1}$ be the restriction of $\D$ to $I_{\le 1}$ and $\D_{\le 1}^\circ$ the one to $I_{\le 1}^\circ$.

We have the following observations. We use the notion of total cofiber of a diagram with respect to
a subdiagram, see \cite{huettemann}. This is defined to be the cofiber of the natural map from
the homotopy colimit of the subdiagram to the homotopy colimit of the total diagram. Usually this
can be viewed as the total object corresponding to a diagram of a certain shape, e.g. of a cubical diagram.
For a functor to be homotopy right cofinal see \cite[Definition 19.6.1]{hirschhorn}.

\begin{lemma}\label{cats}
 \begin{enumerate}
 \item The inclusion $I_{\le 1}^\circ \to I^\circ$ is homotopy right cofinal.
 \item The total cofiber of the diagram $\D_{\le 1}$ with respect to the inclusion $\D_{\le 1}^\circ \to \D_{\le 1}$, i.e. the cofiber of the map $\hocolim \D_{\le 1}^\circ
       \to \D(1)=\MGL$, is isomorphic to $\Q$ in $\caD_\MGL$.
\end{enumerate}

\end{lemma}
\begin{proof}
 (1): Let the inclusion be denoted by $j$. We have to show that for any
object $o \in I^\circ$ the
under category $o \backslash j$ is contractible. But this under
category has the initial object $(o',o \to o')$, where $o'$ contains $x_i$ with
multiplicity $1$ if $x_i | o$, otherwise it contains it with multiplicity $0$.

 (2): We first show the analogous statement for finitely many $x_i$: the total cofiber of $\D|_{\einscat^n}$ with
respect to the inclusion $(\einscat^n)^\circ \subset \einscat^n$ is equivalent to
the quotient $\MGL/(x_1,\ldots,x_n) \MGL$.
This is proved by induction on $n$ using the following statement (with $J=\einscat^n$ for the induction step $n \mapsto n+1$):

Let $J^\circ$ be a small category and let $J$ be the same category added a terminal object (so if $J$ had already a terminal object
this object will no longer be terminal). Let $\caC$ be a pointed model category and $G \colon J \times \einscat \to \caC$ be a functor. Let $(J \times \einscat)^\circ$ be again the
category obtained by removing the terminal object. Then the total cofiber of $G$ with respect to the inclusion $(J \times \einscat)^\circ \to J \times \einscat$
can be computed as follows: it is the cofiber of the map $\Totcof(G|_{J \times \{0\}}) \to \Totcof(G|_{J \times \{1\}})$, $\Totcof$
denoting the total cofiber with respect to the inclusions of the respective subcategories obtained by removing the terminal object.

By considering pointed mapping spaces this statement reduces to the dual statement for pointed simplicial sets. Here to
compute homotopy limits we can use the injective model structure on diagram categories. Now we switch and let $J$ be the opposite of the original $J$,
similarly with $J^\circ$. The category $(J \times \einscat)^\circ$ this time denotes $J \times \einscat$ with the initial object removed. Let $G \colon J \times \einscat \to \sSet_\bullet$ be a functor.
Fix an injectively fibrant replacement $G \to R$. We note that $R|_{(J \times \einscat)^\circ}$, $R|_{J^\circ \times \{i\}}$, $i=0,1$,
are again injectively fibrant: the respective restriction functors are right Quillen functors since the corresponding adjoints preserve objectwise cofibrations.

Let $\I$ be the initial object of $J$.
We replace the diagram $R$ by the square
$$\xymatrix{R(\I,0) \ar[r] \ar[d] &
\lim R|_{J^\circ \times \{0\}} \ar[d] \\
R(\I,1) \ar[r] & \lim R|_{J^\circ \times \{1\}}}.$$
This yields again an injectively fibrant diagram, the limit of the lower
right triangle gives the homotopy limit of $G|_{(J \times \einscat)^\circ}$,
whence the statement follows from lemma (\ref{pullb-fib}).

We are left to prove the statement for infinitely many $x_i$.
The restriction functors from $I_{\le 1}^\circ$-diagrams to $(\einscat^n)^\circ$-diagrams
preserve projectively cofibrant diagrams, whence $\hocolim \D_{\le 1}^\circ
\simeq \hocolim_n (\hocolim \D|_{(\einscat^n)^\circ})$. This shows the claim.
\end{proof}

For any $n > 0$ let $I_{\deg \ge n}$ be the subcategory of $I$ of monomials of degree $\ge n$, where the degree of a monomial
$x_1^{k_1} \cdots x_n^{k_n}$ is $\sum_{i=1}^n i \cdot k_i$. Moreover for a monomial $M$ let $I_{\ge M}$ be the subcategory
of all monomials which are divisible by $M$. We also let $\D_{\deg \ge n} = \D|_{I_{\deg \ge n}}$ and
$\D_{\ge M} = \D|_{I_{\ge M}}$.

\begin{proposition}\label{cats2}
 Let $F \colon I_{\deg \ge n} \to \caC$ be a diagram in a cofibrantly generated model category $\caC$ such that for
any monomial $M$ of degree $n$ the natural map $\hocolim F|_{I_{\ge M}} \to F(M)$ is an equivalence.
Then the natural map $$\hocolim (F|_{I_{\deg \ge n+1}}) \to \hocolim F$$ is an equivalence.
\end{proposition}
\begin{proof}
Let $Q\to F$ be a cofibrant replacement of $F$ for the
projective model structure on $\caC^{I_{\deg \ge n}}$.
Let $M$ be a monomial of degree $n$. We claim $Q|_{I_{\ge n}}$
is still cofibrant: in fact the right adjoint $r$ to the restriction
functor preserves objectwise fibrations: for $o \in I_{\ge M}$
we have $r(H)(o)=H(o)$, for $o \notin I_{\ge M}$
we have $r(H)(o)=\mathrm{pt}$.

Thus if we define $Q'$ by replacing for any $M$ with $\deg(M)=n$ the object
$Q(M)$ with $\colim Q|_{I_{\ge M}}$ and leaving the other
entries unchanged we do not change the weak
homotopy type of $Q$. As above $Q'|_{I_{\deg \ge n+1}}$
is cofibrant, moreover $Q'$ is cofibrant itself
since for any $B \in \caC^{I_{\deg \ge n}}$ we have
$\Hom(Q'|_{I_{\deg \ge n+1}},B|_{I_{\deg \ge n+1}})
\cong \Hom(Q',B)$. Thus $\hocolim Q'|_{I_{\deg \ge n+1}} \simeq
\colim Q'|_{I_{\deg \ge n+1}} \cong
\colim Q' \simeq \hocolim Q' \simeq
\hocolim Q$, which finishes the proof.
\end{proof}

For the proof of the next statement we use the notion of a left
Quillen presheaf on $\omega$, where $\omega$ is the first infinite ordinal.
In this case such a presheaf is given by a model category $\caC_n$ for each natural
number $n$ and left Quillen functors $f_n \colon \caC_{n+1} \to \caC_n$ for each $n \ge 0$.
A section consists of objects $X_n \in \caC_n$ for each $n \ge 0$ and maps
$f_n(X_{n+1}) \to X_n$, $n \ge 0$. It is called homotopy cartesian if the
maps $(\bL f_n)(X_{n+1}) \to X_n$ are isomorphisms in $\Ho (\caC_n)$. The category of sections possesses
the invers model structure where weak equivalences and cofibrations are objectwise.
We will use the fact that the mapping space out of a homotopy cartesian section $X_\bullet$
into any $Y_\bullet$ is given as the homotopy limit over the individual mapping spaces $\mathrm{map}(X_n,Y_n)$.

\begin{lemma}\label{stab-pres}
 Let $F \colon \caC \leftrightarrow \caD \colon G$ be a Quillen adjunction between stable pointed model
categories. Suppose $\bR G \colon \Ho \caD \to \Ho \caC$ preserves sums.
Then $\bR G$ preserves homotopy colimits.
\end{lemma}
\begin{proof}
We indicate a proof. First note that a homotopy colimit of
a functor $F \colon I \to \caD$, which we suppose to take values in cofibrant objects, can be computed
as the homotopy colimit of the simplicial diagram $[n] \mapsto \coprod_{\varphi \colon [n] \to I} F(\varphi(0))$,
$[n]$ the ordered set $\{0,\ldots,n\}$ viewed as a category.
By considering mapping spaces this reduces to the statement that
mapping spaces in $\sSet^I$ can be computed by homotopy ends.

Since $\bR G$ preserves homotopy coproducts
by assumption it thus suffices
to show that colimits over $\bigtriangleup^\op$ are preserved.
Let $\alpha \colon \bigtriangleup^\op \to \caD$ be a functor. We claim
$\hocolim (\alpha) \simeq \hocolim_n \hocolim(\alpha|_{\bigtriangleup_{\le n}^\op})$.
This proves the above since $\bR G$ preserves finite homotopy colimits since
we are dealing with a stable situation and sequential
homotopy colimits over $\omega$ since $\bR G$ preserves sums.
Using mapping spaces we reduce the statement about $\alpha$
to the dual statement in simplicial sets. Now we observe that $\Ho (\sSet^\bigtriangleup)$
is equivalent to homotopy cartesian sections in the homotopy category of the category of sections
of the left Quillen presheaf on $\omega$, $n \mapsto \sSet^{\bigtriangleup_{\le n}}$, the categories
$\sSet^{\bigtriangleup_{\le n}}$ carrying the injective model structure such that the restriction maps
preserve cofibrations. Considering mapping spaces out of constant diagrams shows the claim
since a mapping space between homotopy cartesian sections is a homotopy limit over $\omega$
of the individual mapping spaces.
\end{proof}

We now turn to the functors $f_i \colon \SH(S) \to \Sigma_T^i \SH(S)^\eff \subset \SH(S)$ introduced in
\cite{voe-slice}. These are defined as right adjoints to the inclusions $\Sigma_T^i \SH(S)^\eff \subset \SH(S)$.
These functors can be defined on the level of model categories by using
colocalization of model categories, see \cite{pelaez}. In particular it makes sense
to ask whether these functors preserve homotopy colimits.

\begin{corollary}\label{f-stab-pres}
 The functors $f_i$ preserve homotopy colimits.
\end{corollary}
\begin{proof}
 This follows from lemma (\ref{stab-pres}) and the fact that the $f_i$ preserve sums \cite{voe-slice}.
\end{proof}

We now can state the main observation of this text.

\begin{theorem}\label{rel-voe-conj}
 Suppose the natural map $\MGL \to \Q$ is the map from $\MGL$ to
 its zero-slice. Then \cite[Conjecture (5)]{voe-slice} holds with $H_{\MU_{2q}}$ replaced by
$s_0 \MGL \otimes \MU_{2q}$.
\end{theorem}

Note the compatibility with the natural homomorphism $\MU_* \to \MGL_{**}$ still makes sense.
\begin{proof}
Lemma (\ref{cats}).(1) and (2) and \cite[Theorem 19.6.13]{hirschhorn}
imply that the map from $\MGL$ to the total cofiber
of $\D$ with respect to $\D^\circ \to \D$ is the map to $\Q$, hence by assumption this
is the map to the zero-slice.
Thus $\hocolim \D^\circ \to f_1\MGL$ is an equivalence. By corollary
(\ref{f-stab-pres}) the $f_i$ commute with homotopy colimits.
We denote by $\caF_i$ functorial versions of the $f_i$ on the level of model categories.
Then we can rewrite $f_2 \hocolim \D^\circ \simeq \hocolim \caF_2 \D^\circ$.
Observe the diagram $\caF_2 \D^\circ$ satisfies the asssumptions of proposition
(\ref{cats2}) for $n=1$, hence $f_2 \MGL \simeq \hocolim \caF_2 \D^\circ \simeq \hocolim \D_{\deg \ge 2}$.

Increasing the degree of the monomials it follows from proposition
(\ref{cats2}) by induction that
$f_n \MGL \simeq \hocolim \D_{\deg \ge n}$. In addition the maps
$$f_{n+1} \MGL \to f_n \MGL$$ are the naturally induced maps
$$\hocolim \D_{\deg \ge n+1} \to\hocolim \D_{\deg \ge n}.$$
Thus the $s_n \MGL$ are the cofibers of these maps.
We rewrite the source again as $\hocolim \caF_{n+1} \D_{\deg \ge n}$.
Cofibers commute with homotopy colimits, thus $s_n \MGL \simeq \hocolim \, \cofib(\caF_{n+1} \D_{\deg \ge n}\to
\D_{\deg \ge n})$. The value of this cofiber of diagrams at a monomial of degree $> n$ is zero,
at a monomial of degree $n$ we exactly have $\Sigma_T^n s_0 \MGL$. Now
it is easy to see that the homotopy colimit of such a diagram is the homotopy coproduct of
the entries in degree $n$, which is $\Sigma_T^n s_0 \MGL \otimes \MU_{2n}$.

The compatibility with the natural homomorphism $\MU_* \to \MGL_{**}$ follows by the choice of the $x_i$.

\end{proof}

Recall the natural orientation $\MGL \to \MZ$ of the motivic Eilenberg MacLane spectrum.
It is additive, in particular the $x_i$ all map to zero in $\MZ_{**}$. Iteratively we get
factorizations $\MGL/(x_1,\ldots,x_n) \MGL \to \MZ$
in $\SH(S)$ in a compatible way, which gives a map $\Q \to \MZ$.

\begin{corollary} \label{per-field}
 Suppose $S$ is the spectrum of a perfect field. If the natural map $\Q\to \MZ$ is an
isomorphism then \cite[Conjecture (5)]{voe-slice} holds.
\end{corollary}
\begin{proof}
 By \cite{levine-htp} and \cite{voevodsky-zero-slice} $s_0 \unit \cong \MZ$. Thus by corollary (\ref{s0-isom})
the map $\MGL \to \MZ$ is the map from $\MGL$ to its zero-slice. The statement follows from theorem (\ref{rel-voe-conj}).
\end{proof}

\section{$K$-theory}

By the Landweber exactness theorem \cite{NSO1} and \cite{NSO2} the spectrum $\KGL$
is the Landweber spectrum associated to the $\MU_*$-algebra $x_1^{-1}\MU_*/(x_2,x_3,\ldots) \MU_*
\cong \integers[u,u^{-1}]$. The latter algebra classifies the multiplicative formal group law
$x + y -uxy$.

The orientation $\MGL\to \KGL$ sends all $x_i \in \MGL_{2i,i}$, $i \ge 2$, to $0 \in \KGL_{2i,i}$.
Thus we obtain a factorization $\MGL/(x_2,x_3\ldots)\MGL \to \KGL$.
Since $x_1$ acts invertibly on $\KGL$ this map further factors as $$B \colon x_1^{-1} \MGL/(x_2,x_3\ldots)\MGL \to \KGL.$$

\begin{lemma}\label{land-seq}
Let $$0 \to M_* \overset{f}{\to} N_* \to Q_* \to 0$$ be a short
exact sequence of evenly graded Landweber exact $\MU_*$-modules.
Let $F \colon \EE_M \to \EE_N$ be a map of spectra in $\SH(\Z)_\caT$
representing the homology transformation given by $f$
via the motivic Landweber exact functor theorem. Then
the cofiber of $F$ represents the Landweber theory given by $Q_*$.
\end{lemma}
\begin{proof}
We let $\X$ be the stack $[\Spec (\MU_*)/\Spec (\MU_* \MU)]$
and $\widetilde{M}_*$ etc. the quasi coherent sheaves on $\X$
obtained from the $M_*$ etc. by pushforward along the
map $\Spec (\MU_*) \to \X$. Then by Landweber exactness
the sequence
$$0 \to \widetilde{M}_* \to \widetilde{N}_* \to \widetilde{Q}_* \to 0$$
is a short exact sequence of flat $\caO_\X$-modules.
In particular tensoring this sequence with a quasi coherent
$\caO_\X$-module yields again a short exact sequence.
However, the Landweber theorem (\cite{NSO1}) is proved by considering
the $\MGL$-homology of a motivic spectrum as a
$(\MGL_*,\MGL_*\MGL)$-comodule and via restriction
along a map of Hopf algebroids as $(\MU_*,\MU_*\MU)$-comodule, then tensoring this
over $\caO_\X$ with $\widetilde{M}_*$ and finally
pulling back to $\Spec (\MU_*)$. This shows
that the sequence of motivic homology theories
obtained from the original sequence is short exact.
In particular if the first map is represented
by any map of motivic spectra then the cofiber will
represent the homology theory associated with $Q_*$.
\end{proof}

\begin{theorem}\label{quot-KGL}
The map of spectra $B \colon x_1^{-1} \MGL/(x_2,x_3,\ldots)\MGL \to \KGL$
is an isomorphism.
\end{theorem}
\begin{proof}
We use that $x_1^{-1} \MGL/(x_2,x_3,\ldots)\MGL$ can also be constructed
by first inverting $x_1$ and then quotienting out the $x_i$, $i \ge 2$.
Now observe that all the quotients $x_1^{-1} \MU_*/(x_2,\ldots,x_n) \MU_*$
are Landweber exact: they are torsionfree, and for a prime $p$
the element $v_1$ of $(\MU_*)_{(p)}$ already acts invertibly.
Thus the
$$0 \to \Sigma^{n+1} x_1^{-1} \MU_*/(x_2,\ldots,x_n)\MU_* \to x_1^{-1} \MU_*/(x_2,\ldots,x_n) \MU_*$$
$$\to x_1^{-1} \MU_*/(x_2,\ldots,x_{n+1}) \MU_* \to 0$$
($\Sigma$ refers to a shift of evenly graded groups)
are short exact sequences of evenly graded Landweber modules,
and lemma (\ref{land-seq}) applies. Thus it follows inductively that the
quotients $_1^{-1} \MGL/(x_2,\ldots,x_n)\MGL$
are the Landweber spectra for the modules $x_1^{-1} \MU_*/(x_2,\ldots,x_n)
\MU_*$. Passing to the colimit shows the claim.
\end{proof}

\begin{remark}
Fix a prime $p$ and let $\BP^\Top$ be the topological Brown-Peterson spectrum.
Let $\BP$ be the motivic spectrum on the Landweber coefficients $\BP^\Top_*$.
It can be seen that this coincides with the definition given in \cite{vezzosi}
since both definitions give rise to the universal oriented ring cohomology theory
on compact objects with $p$-typical formal group law.

Let $\EE(n)^\Top$ be the topological spectrum for the coefficients
$$v_n^{-1} \BP^\Top_*/(v_{n+1},v_{n+2},\ldots).$$
The corresponding motivic Landweber spectrum si denoted $\EE(n)$.
Now the quotients $v_n^{-1} \BP^\Top_*/(v_{n+1},\ldots,v_{n+k}) \BP^\Top_*$
are also Landweber exact. Hence by the same method as for theorem (\ref{quot-KGL})
we see that there is an isomorphsim
$$\EE(n) \cong v_n^{-1} \BP/(v_{n+1},v_{n+2},\ldots)\BP,$$
compare with \cite{hornbostel}.

The map from \cite[par. 6]{spitzweck-slice}
is ill-defined since $\BP$ is not a direct summand of $\MGL_{(p)}$
as $\MGL$-module.
\end{remark}

Next we analize the relationship to connective or effective $K$-theory.
A version of that has been introduced in \cite{levine-htp}.

Since $\MGL/(x_2,\ldots)\MGL$ is effective, the canonical map
$\MGL/(x_2,\ldots)\MGL \to x_1^{-1} \MGL/(x_2,\ldots)\MGL \cong \KGL$
factors as $$B^\eff \colon \MGL/(x_2,\ldots)\MGL \to f_0 \KGL.$$

\begin{proposition}
Suppose the map
$\MGL \to \Q$ is the map from $\MGL$ to its zero-slice.
 Then $B^\eff$ induces isomorphisms on slices.
If $S$ is the spectrum of a perfect field and $B^\eff$ is an isomorphism,
then the map $\MGL \to \Q$ is the map from $\MGL$ to its zero-slice.
\end{proposition}
\begin{proof}
Let us assume $\MGL \to \Q$ is the map to the zero-slice.
Then by theorem (\ref{rel-voe-conj}) the assumption $\SlMGL$
of \cite{spitzweck-slice} is fulfilled,
hence by \cite[corollary 4.2]{spitzweck-slice} we have
$s_0 \KGL \cong s_0 \MGL$, and the isomorphism
is realized by the map $s_0 \MGL \to s_0 f_0 \KGL$.

By the periodicity of $\KGL$ the map $f_1 \KGL \to f_0 \KGL$
is isomorphic to the map $\Sigma_T f_0 \KGL \to f_0 \KGL$ given by
multiplication by the Bott element $u \in (f_0 \KGL)_{2,1} = \KGL_{2,1}$.
Thus the cofiber of this map is isomorphic to $s_0 \MGL$,
and the same holds true for the cofiber of the map
$\Sigma_T \MGL/(x_2,\ldots)\MGL \to \MGL/(x_2,\ldots)\MGL$
given by multiplication by $x_1$  by the assumption.
Interating this process shows the claim.

Suppose now that $S$ is the spectrum of a perfect field.
Then by \cite{levine-htp} there are isomorphisms $s_0 \unit \cong s_0 \KGL \cong \MZ$,
and by corollary (\ref{s0-isom}) we also have $s_0 \MGL \cong s_0 \KGL$.
Now if $B^\eff$ is an isomorphism, then the map $\Sigma_T \MGL/(x_2,\ldots)\MGL \to \MGL/(x_2,\ldots)\MGL$
given by multiplication by $x_1$ is the map $f_1 \MGL/(x_2,\ldots)\MGL \to \MGL/(x_2,\ldots)\MGL$
showing the claim.
\end{proof}

\section{Rational slices}

We show in this paragraph that the asumptions from the
last sections hold true after rationalization, at least over
regular base schemes.

We denote by $\LQ$ the Landweber spectrum associated to the $\MU_*$-module $\bQ$
carrying the additive formal group law.

We note that any rational motivic Landweber spectrum has a decomposition
into a sum of $\Sigma^{2i,i} \LQ$ for various $i$. This follows directly
from the corresponding decomposition of the topological coefficients.

Since the Landweber coefficients for the rational $K$-theory spectrum
are $\bQ[v,v^{-1}]$ ($v$ the Bott element in homological degree $2$) we obtain
\begin{equation}\label{kgl-decomp1}
\KGL_\bQ \cong \PLQ = \bigoplus_{i \in \integers} \Sigma^{2i,i} \LQ,
\end{equation}
compare also with \cite[Theorem 10.1]{NSO1}.

Using projectors \cite[Theoreme IV.72]{riou-these} gives a decomposition
\begin{equation}\label{kgl-decomp2}
 \KGL_\bQ \cong \bigoplus_{i \in \integers} \bH_B^{(i)}.
\end{equation}
This decomposition is first defined for regular base schemes $S$. For morphisms
between regular base schemes it pulls back. It thus makes sense to
pull back the decomposition given over $\Spec(\integers)$ to any base scheme,
which we shall assume.

The next statement follows from the structure of a rational Snaith map
$$\colim_n (\LQ \wedge \Sigma^{-2n,-n} \Sigma^\infty \PP^\infty_+)
\to \KGL_\bQ$$ (see \cite{spitzweck-oestvaer} for the Snaith map) and the way the projectors
for (\ref{kgl-decomp2}) are defined.

We call the map $$\Sigma^{-2i,-i} \Sigma^\infty \PP^\infty_+
\to \KGL_\bQ$$ the $i$'th Snaith map.
\begin{proposition}\label{decomp}
 The decompositions (\ref{kgl-decomp1}) and (\ref{kgl-decomp2}) coincide.
\end{proposition}
\begin{proof}
We can assume the base is $\Spec(\integers)$.
We first compute the map
\begin{equation} \label{pr-map}
\bQ^\integers \to \prod_{j \in \integers} \Hom(\Sigma^{2j,j} \LQ,
\Sigma^{2j,j} \LQ) \to \Hom(\PLQ,\PLQ) \cong
\end{equation}
$$\Hom(\KGL_\bQ,\KGL_\bQ) \to \Hom(\Sigma^{-2i,-i} \Sigma^\infty \PP^\infty_+,
\KGL_\bQ)=(\KGL_\bQ^{**} \poauf x \pozu)^{2i,i} = \bQ \poauf u \pozu.$$
The first map in the second line is precomposition with the $i$'th Snaith map.
On the right hand side an appropriate Bott shift is applied to powers of $x$ to obtain
the powers of $u$.
Note that for the system in $i$ on the right hand side
all transition maps are the same and given by $u^j \mapsto -ju^{j-1} + j u^j$ (this follows
from the multiplicative formal group law for $\KGL$).
Note also that we define the Bott element to be the negative reduced class of
$\PP^1$ in $\PP^\infty$ and $\BGL$.

We let $\Hom(\Sigma^\infty \PP^\infty_+,\PLQ)=\bQ \poauf u' \pozu$. Here $u'$
is the generator for the additive orientation on $\PLQ$ times the Bott element.

The multiplicative isomorphism $$\bQ \poauf u \pozu \cong \Hom(\Sigma^\infty \PP^\infty_+,\KGL_\bQ)
\cong \Hom(\Sigma^\infty \PP^\infty_+,\PLQ) \cong \bQ \poauf u \pozu$$ is given
by $u = 1 - e^{u'}$, which is the change of formal parameters between the multiplicative
and additive formal group law.

Moreover the zero'th Snaith map to $\KGL_\bQ$ is the element $1-u \in \bQ \poauf u \pozu
\cong \Hom(\Sigma^\infty \PP^\infty_+,\KGL_\bQ)$. Thus the Snaith map translates to
$$e^{u'} \in \bQ \poauf u' \pozu \cong \Hom(\Sigma^\infty \PP^\infty_+,\PLQ) \cong
\Hom_{\Mod_{\SH(\integers)}(\LQ)}(\LQ \wedge \Sigma^\infty \PP^\infty_+,\PLQ).$$

We let $\LQ \wedge \Sigma^\infty \PP^\infty_+ = \LQ<b_0,b_1,b_2,\ldots >$, where the $b_i$
are dual generators to the powers of the orientation generator.
It follows that after these identifications the Snaith map $$\LQ <b_0,b_1,b_2,\ldots > \to
\PLQ$$ has the effect $$\LQ b_j \overset{\frac{1}{j!}}{\to} \Sigma^{2j,j} \LQ$$
on summands.

The Snaith map $\Sigma^{-2i,i} \Sigma^\infty \PP^\infty_+ \to \PLQ$ is just
the same map with the appropriate Bott shift applied and is thus given on summands by
$$\Sigma^{-2i,-i} \LQ b_j \overset{\frac{1}{j!}}{\to} \Sigma^{-2i+2j,-i+j} \LQ.$$

Now let us start with the $n$'th projector $p \in \Hom(\PLQ,\PLQ)$, $n \in \integers$.
Composition with the $i$'th Snaith map gives us a map
$$\Sigma^{-2i,-i} \LQ<b_0,b_1,b_2,\ldots > \to \PLQ$$
sending, if $n \ge -i$, $\Sigma^{-2i,-i}b_{n+i}$ to $\frac{1}{(2n+2i)!} \Sigma^{2n,n} \LQ$
(abusing notation) and the other generators to $0$. Thus this map
is the element $$\frac{1}{(2n+2i)!} {u'}^{2n+2i} \in \bQ \poauf u' \pozu =
\Hom(\Sigma^{-2i,-i}\Sigma^\infty \PP^\infty_+,\PLQ).$$

Applying the equation $u'=\ln(1-u)$ we get the result that the map
(\ref{pr-map}) sends the projector $p$, if $n \ge -i$, to
$$\frac{1}{(2n+2i)!} {\ln(1-u)}^{2n+2i} \in \bQ \poauf u \pozu =
\Hom(\Sigma^{-2i,-i}\Sigma^\infty \PP^\infty_+,\KGL_\bQ),$$
otherwise to $0$.

This is the same formula as used in \cite{riou-these} to define the projectors,
see \cite[Def. IV.62, after Cor. IV.67, Def. IV.71]{riou-these}.

Now it suffices to observe that the Snaith system is indeed the system appearing
in loc. cit. to describe endomorphisms of $\KGL$ and $\KGL_\bQ$, see
\cite[Th. IV.13, before Prop. IV.33]{riou-these}.
\end{proof}

\begin{lemma} \label{LQQ}
 The natural map $\MGL_\bQ \to \LQ$ factors as an isomorphism $\Q_\bQ \to \LQ$.
\end{lemma}
\begin{proof}
The claim follows from the decompositions
$$\MGL_\bQ \cong \LQ[b_1,b_2,\ldots] \cong \LQ[x_1,x_2,\ldots]$$
(see \cite{NSO1} for the definition of the $b_i$, see also \cite[Theorem 10.5]{NSO1}).
\end{proof}

\begin{proposition} \label{rat-slice}
 Suppose $S$ is regular. Then the map $\MGL_\bQ \to \LQ$ is the map to the zero slice.
\end{proposition}
\begin{proof}
 First note that by the decomposition $\MGL_\bQ \cong \LQ[b_1,b_2,\ldots]$
and the effectivity of $\MGL$ it follows that $\LQ$ is effective.
By lemma (\ref{LQQ}) we have to show that $\MGL \to \LQ$ is the map to the zero slice.
The decomposition of $\MGL_\bQ$ shows that for this it suffices to show that
$\Hom_{\SH(S)}(\Sigma^{p,q} \Sigma^\infty_+ X,\LQ)=0$ for $X \in \Sm/S$, $p \in \integers$, $q \ge 1$.
Equivalently we have to show that every map $\Sigma^{p,q} \Sigma^\infty_+ X \to \KGL_\bQ$,
$X \in \Sm/S$, $p \in \integers$ ,$q \ge 0$,
factors through $\LQ[u]$ with respect to the decomposition
$\KGL_\bQ \cong \LQ[u,u^{-1}]$. We can assume $q=0$ by replacing $X$ by $X \times \bG_m^q$.
Proposition (\ref{decomp}) and \cite[Corollaire VI.75]{riou-these} imply this for non-negative $p$. Suppose $p < 0$. Then by
the periodicity of $\PLQ$ the claim we want to show is equivalent to the statement
that for $p+2q \ge 0$ every map $\Sigma^{p+2q,q} \Sigma^\infty_+ X \to \LQ[u,u^{-1}]$ factors through
$u^q\LQ[u]$. This follows from the decomposition \cite[Corollary (5.5)(ii)]{srinivas} of
the algebraic $K$-theory of the Laurent polynomials over a regular ring.
\end{proof}

\begin{corollary} Suppose $S$ is regular.
Then $s_i(\MGL_\bQ) \cong \Sigma_T^i \LQ \otimes \MU_{2i}$ compatible
with the homomorphism $\MU_* \to \MGL_{\bQ,**}$.
\end{corollary}

\begin{proof}
By proposition (\ref{rat-slice}) we have $s_0(\LQ)=\LQ$.
The claim follows from the decomposition $\MGL_\bQ \cong \LQ[x_1,x_2,\ldots]$.
\end{proof}

\begin{corollary}
Suppose $S$ is regular. Then $s_0(\unit
_\bQ)=s_0(\unit)_\bQ=\LQ$.
\end{corollary}
\begin{proof}
This follows now from Corollary (\ref{s0-isom}).
\end{proof}

\bibliographystyle{plain}
\bibliography{slmgl}

\begin{thebibliography}{10}

\bibitem{dugger-isaksen}
Daniel Dugger and Daniel~C. Isaksen.
\newblock Motivic cell structures.
\newblock {\em Algebr. Geom. Topol.}, 5:615--652 (electronic), 2005.

\bibitem{ekmm}
A.~D. Elmendorf, I.~Kriz, M.~A. Mandell, and J.~P. May.
\newblock {\em Rings, modules, and algebras in stable homotopy theory},
  volume~47 of {\em Mathematical Surveys and Monographs}.
\newblock American Mathematical Society, Providence, RI, 1997.

\bibitem{hirschhorn}
Philip~S. Hirschhorn.
\newblock {\em Model categories and their localizations}, volume~99 of {\em
  Mathematical Surveys and Monographs}.
\newblock American Mathematical Society, Providence, RI, 2003.

\bibitem{hornbostel}
Jens Hornbostel.
\newblock Localizations in motivic homotopy theory.
\newblock {\em Math. Proc. Cambridge Philos. Soc.}, 140(1):95--114, 2006.

\bibitem{hovey-book}
Mark Hovey.
\newblock {\em Model categories}, volume~63 of {\em Mathematical Surveys and
  Monographs}.
\newblock American Mathematical Society, Providence, RI, 1999.

\bibitem{huettemann}
Thomas H{\"u}ttemann.
\newblock Total cofibres of diagrams of spectra.
\newblock {\em New York J. Math.}, 11:333--343 (electronic), 2005.

\bibitem{mot-htp}
Tylor Lawson.
\newblock Motivic homotopy.
\newblock Notes of a course given by Mike Hopkins and Marc Levine at Harvard,
  http://www.math.umn.edu/~tlawson/motivic.html.

\bibitem{levine-comparison}
Marc Levine.
\newblock Comparison of cobordism theories.
\newblock Preprint, arXiv:0807.2238.

\bibitem{levine-htp}
Marc Levine.
\newblock The homotopy coniveau tower.
\newblock {\em J. Topol.}, 1(1):217--267, 2008.

\bibitem{NSO2}
Niko Naumann, Markus Spitzweck, and Paul~Arne {\O}stv{\ae}r.
\newblock Chern classes, {$K$}-theory and {L}andweber exactness over nonregular
  base schemes.
\newblock Preprint, arXiv:0809.0267.

\bibitem{NSO1}
Niko Naumann, Markus Spitzweck, and Paul~Arne {\O}stv{\ae}r.
\newblock Motivic {L}andweber exactness.
\newblock Preprint, arXiv 0806.0274.

\bibitem{pelaez}
Pablo Pelaez.
\newblock Multiplicative properties of the slice filtration.
\newblock PhD thesis, arXiv 0806.1704.

\bibitem{riou-these}
Jo{\"e}l Riou.
\newblock Op\'erations sur la {$K$}-th\'eorie alg\'ebrique et r\'egulateurs via
  la th\'eorie homotopique des sch\'emas.
\newblock Thesis, K-theory Preprint Archives, 793.

\bibitem{spitzweck-thesis}
M.~Spitzweck.
\newblock {\em Operads, Algebras and Modules in Model Categories and Motives}.
\newblock PhD thesis, University of Bonn, 2001.

\bibitem{spitzweck-slice}
Markus Spitzweck.
\newblock Slices of motivic {L}andweber spectra.
\newblock Preprint, arXiv:0805.3350v1.

\bibitem{spitzweck-oestvaer}
Markus Spitzweck and Paul~Arne {\O}stv{\ae}r.
\newblock The {B}ott inverted infinite projective space is homotopy
  {$K$}-theory.
\newblock Preprint, http://folk.uio.no/paularne/bott.pdf.

\bibitem{srinivas}
V.~Srinivas.
\newblock {\em Algebraic {$K$}-theory}, volume~90 of {\em Progress in
  Mathematics}.
\newblock Birkh\"auser Boston Inc., Boston, MA, second edition, 1996.

\bibitem{vezzosi}
Gabriele Vezzosi.
\newblock {B}rown-{P}eterson spectra in stable {$\Bbb A\sp 1$}-homotopy theory.
\newblock {\em Rend. Sem. Mat. Univ. Padova}, 106:47--64, 2001.

\bibitem{voevodsky-zero-slice}
V.~Voevodsky.
\newblock On the zero slice of the sphere spectrum.
\newblock {\em Tr. Mat. Inst. Steklova}, 246(Algebr. Geom. Metody, Svyazi i
  Prilozh.):106--115, 2004.

\bibitem{voe-slice}
Vladimir Voevodsky.
\newblock Open problems in the motivic stable homotopy theory. {I}.
\newblock In {\em Motives, polylogarithms and Hodge theory, Part I (Irvine, CA,
  1998)}, volume~3 of {\em Int. Press Lect. Ser.}, pages 3--34. Int. Press,
  Somerville, MA, 2002.

\end{thebibliography}

\begin{center}
Fakult{\"a}t f{\"u}r Mathematik, Universit{\"a}t Regensburg, Germany.\\
e-mail: Markus.Spitzweck@mathematik.uni-regensburg.de
\end{center}

\end{document}